\documentclass[a4paper,12pt]{article}
\usepackage{amsmath,amsthm,amssymb}

\newtheorem{theorem}{Theorem}
\newtheorem{defn}{Definition}

\newtheorem{lemma}{Lemma}
\newtheorem{proposition}[lemma]{Proposition}
\newtheorem{remark}{Remark}

\newcommand{\C}{\mathbb{C}}
\newcommand{\R}{\mathbb{R}}

\newcommand{\Z}{\mathbb{Z}}
\newcommand{\dual}[2]{\langle #1, #2\rangle}
\newcommand{\pd}{\partial}
\newcommand{\even}{{\rm even}}
\newcommand{\eps}{\varepsilon}
\newcommand{\reP}{{\rm Re}\,}
\newcommand{\imP}{{\rm Im}\,}

\begin{document}

\begin{center}
\large{\textbf{
Instability of standing waves for a system of nonlinear Schr\"odinger 
equations in a degenerate case}}
\end{center}

\vspace{5mm}

\begin{center}
{Shotaro Kawahara} \hspace{1mm} and \hspace{1mm}
{Masahito Ohta}
\end{center}
\begin{center}
Department of Mathematics, 
Tokyo University of Science, \\
1-3 Kagurazaka, Shinjuku-ku, Tokyo 162-8601, Japan
\end{center}

\begin{abstract}
We study a system of nonlinear Schr\"odinger equations with cubic interactions  
in one space dimension. 
The orbital stability and instability of semitrivial standing wave solutions are studied 
for both non-degenerate and degenerate cases. 
\end{abstract}

\begin{center}
{\small 
Dedicated to Professor Nakao Hayashi on the occasion of his sixtieth birthday} 
\end{center}

\section{Introduction}

In this paper, we study the orbital stability and instability of standing wave solutions 
for the following system of nonlinear Schr\"odinger equations with cubic interactions 
in one space dimension:
\begin{equation} \label{nls}
\begin{cases}
i\pd_t u_1=-\pd_x^2 u_1-\kappa_1 |u_1|^2 u_1-\gamma\,u_2^2\,\overline{u_1}, \\
i\pd_t u_2=-\pd_x^2 u_2-\kappa_2 |u_2|^2 u_2-\gamma\, u_1^2\,\overline{u_2}, 
\end{cases}
\end{equation}
where $u_1$ and $u_2$ are complex-valued functions of $(t,x)\in \R\times \R$, 
and $\kappa_1$, $\kappa_2$ and $\gamma$ are positive constants. 
The system \eqref{nls} appears in various areas of physics 
such as nonlinear optics, Bose-Einstein condensates, and so on 
(see, e.g., \cite{agr, IMW, KS, WC}). 

By the standard theory (see, e.g., \cite[Chapter 4]{caz}), 
the Cauchy problem for \eqref{nls} 
is globally well-posed in the energy space $H^1(\R,\C)^2$, 
and the energy $E$ and the charge $Q$ are conserved, where 
\begin{align*}
&E(\vec u)=\sum_{j=1}^{2}\left(
\frac{1}{2} \|\pd_x u_j\|_{L^2}^2-\frac{\kappa_j}{4} \|u_j\|_{L^4}^4\right)
-\frac{\gamma}{2}\, \reP \int_{\R} u_1^2\,\overline{u_2}^2\,dx, \\
&Q(\vec u)=\frac{1}{2} \sum_{j=1}^{2} \|u_j\|_{L^2}^2
\end{align*}
for $\vec u:=(u_1,u_2)\in H^1(\R,\C)^2$. 
Note that \eqref{nls} is written in a Hamiltonian form $i\pd_t \vec u=E'(\vec u)$, 
and the conservation of charge follows from the invariance of $E$ 
under gauge transform 
$$E(e^{i\theta} \vec u)=E(e^{i\theta}u_1,e^{i\theta}u_2)=E(\vec u)$$
for  $\theta\in \R$ and $\vec u\in H^1(\R,\C)^2$. 

We study the orbital stability and instability of semitrivial standing wave solutions 
$e^{i\omega t} \vec \phi_{\omega}(x)$ for \eqref{nls},
where $\omega>0$ is a constant, 
\begin{equation} \label{ST1}
\vec \phi_{\omega}(x)
:=\left(\frac{1}{\sqrt{\kappa_1}} \varphi_{\omega}(x),0\right), 
\end{equation}
and $\varphi_{\omega}(x)=\sqrt{2\omega} \, {\rm sech} (\sqrt{\omega}\,x)$ 
is a positive and even solution of 
$$-\pd_x^2 \varphi+\omega \varphi-\varphi^3=0, \quad  x\in \R.$$

We are mainly interested in the instability of $e^{i\omega t} \vec \phi_{\omega}(x)$ rather than the stability, 
and we assume the even symmetry for simplicity. 
We denote the set of even functions in $H^1(\R)$ by $H^1_{\even}(\R)$, 
and define $X=H^1_{\even}(\R,\C)^2$. 
Note that $\varphi_{\omega}\in H^1_{\even}(\R)$ and $\vec \phi_{\omega}\in X$. 
Moreover, by the even symmetry of \eqref{nls} 
and the uniqueness of solutions to the Cauchy problem for \eqref{nls}, 
if $\vec u_0\in X$, then the solution $\vec u(t)$ of \eqref{nls} 
with $\vec u(0)=\vec u_0$ satisfies $\vec u \in C(\R,X)$. 

\begin{defn} \label{def1}
We say that the standing wave solution $e^{i \omega t} \vec \phi_{\omega}$ of \eqref{nls} 
is stable if for any $\varepsilon>0$ there exists $\delta>0$ with the following property. 
If $u_0\in X$ satisfies $\|\vec u_0-\vec \phi_{\omega}\|_X<\delta$, 
then the solution $\vec u(t)$ of \eqref{nls} with $\vec u(0)=\vec u_0$ satisfies 
$$\inf_{\theta\in \R}
\|\vec u(t)-e^{i\theta} \vec \phi_{\omega}\|_X<\varepsilon$$ 
for all $t\in \R$. 
Otherwise, $e^{i \omega t} \vec \phi_{\omega}$ is called unstable. 
\end{defn}

We now state our main results in this paper. 

\begin{theorem} \label{thm1}
Let $\kappa_1$, $\kappa_2$, $\gamma$ and $\omega$ be positive constants. 
Then, the semitrivial standing wave solution $e^{i\omega t} \vec \phi_{\omega}(x)$ 
of \eqref{nls} is stable if $\gamma<\kappa_1$, 
and unstable if $\gamma>\kappa_1$. 
\end{theorem}

\begin{theorem} \label{thm2}
Let $\kappa_1$, $\kappa_2$, $\gamma$ and $\omega$ be positive constants, 
and let $\gamma=\kappa_1$. 
Then, the semitrivial standing wave solution $e^{i\omega t} \vec \phi_{\omega}(x)$ 
of \eqref{nls} is stable if $\kappa_2<\kappa_1$, 
and unstable if $\kappa_2>\kappa_1$. 
\end{theorem}

\begin{remark}
By symmetry, 
similar results to Theorems \ref{thm1} and \ref{thm2} also hold 
for semitrivial standing wave solutions of the form
$$e^{i\omega t} \left(0, \frac{1}{\sqrt{\kappa_2}} \varphi_{\omega}(x)\right).$$
\end{remark}

\begin{remark}
For the case $\gamma=\kappa_1=\kappa_2$, 
the system \eqref{nls} has an additional symmetry
\begin{equation*}
\left(\begin{array}{c}
u_1 \\
u_2
\end{array}\right)
\mapsto 
R(\chi)
\left(\begin{array}{c}
u_1 \\
u_2
\end{array}\right), \quad 
R(\chi):=
\left(\begin{array}{cc}
\cos \chi & -\sin \chi \\
\sin \chi & \cos \chi
\end{array}\right)
\end{equation*}
for $\chi\in \R$. 
By this symmetry, 
in the same way as in the proof of Theorem \ref{thm1}, 
we can prove that 
$e^{i\omega t} \vec \phi_{\omega}(x)$ is stable in the following weaker sense. 

For any $\varepsilon>0$ there exists $\delta>0$ with the following property. 
If $u_0\in X$ satisfies $\|\vec u_0-\vec \phi_{\omega}\|_X<\delta$, 
then the solution $\vec u(t)$ of \eqref{nls} with $\vec u(0)=\vec u_0$ satisfies 
$\inf_{\theta, \chi\in \R}
\|\vec u(t)-e^{i\theta} R(\chi) \vec \phi_{\omega}\|_X<\varepsilon$
for all $t\in \R$. 

However, we do not know whether $e^{i\omega t} \vec \phi_{\omega}(x)$ is stable or not 
in the sense of Definition \ref{def1} 
for the case $\gamma=\kappa_1=\kappa_2$. 
\end{remark}

\begin{remark}
The standing waves $e^{i\omega t} \vec \phi_{\omega}(x)$ are also solutions of the following system 
\begin{equation} \label{nls-C}
\begin{cases}
i\pd_t u_1=-\pd_x^2 u_1-\kappa_1 |u_1|^2 u_1-\gamma\, |u_2|^2 u_1, \\
i\pd_t u_2=-\pd_x^2 u_2-\kappa_2 |u_2|^2 u_2-\gamma\, |u_1|^2 u_2. 
\end{cases}
\end{equation}
It is known that 
for any positive constants $\kappa_1$, $\kappa_2$, $\gamma$ and $\omega$, 
the standing wave solution $e^{i\omega t} \vec \phi_{\omega}(x)$ is stable for \eqref{nls-C} 
(see \cite{oht96, NW}). 
\end{remark}

\begin{remark}
For related results on systems of nonlinear Sch\"odinger equations with quadratic interactions, 
see \cite{CO3, HOT}. 
While, for related studies on degenerate cases, 
see \cite{maeda, yamazaki}. 
\end{remark}

The rest of the paper is organized as follows. 
In section \ref{sect2}, we consider the non-degenerate case $\gamma \ne \kappa_1$. 
The stability part of Theorem \ref{thm1} is proved by the standard argument based on \cite{GSS1,wei2}. 
The degenerate case $\gamma=\kappa_1$ is studied in section \ref{sect3}. 
The instability part of Theorem \ref{thm2} is proved by using similar arguments to those in \cite{CO3, oht11}. 

\section{Proof of Theorem \ref{thm1}}\label{sect2}

We regard $L^2(\R,\C)$ as a real Hilbert space with the inner product 
$$(u,v)_{L^2}=\reP \int_{\R}u(x)\overline{v(x)}\,dx,$$
and define the inner products of real Hilbert spaces 
$H=L^2_{\even}(\R,\C)^2$ and $X=H^1_{\even} (\R,\C)^2$ by 
$$(\vec u,\vec v)_{H}=(u_1,v_1)_{L^2}+(u_2,v_2)_{L^2}, \quad 
(\vec u,\vec v)_{X}=(\vec u,\vec v)_{H}+(\pd_x \vec u,\pd_x \vec v)_{H}.$$

For $\omega>0$, we define 
$S_{\omega}(\vec v)=E(\vec v)+\omega Q(\vec v)$ for $\vec v\in X$. 
Then, we have $S_{\omega}'(\vec \phi_{\omega})=0$. 
Moreover, for $a\in \R$, we define $L_a$ by 
$$L_a u=-\pd_x^2 u+\omega u-a \varphi_{\omega}(x)^2 u$$
for $u\in H^1_{\even}(\R,\R)$. 
Then, for $\vec v=(v_1,v_2)\in X$, we have 
\begin{align}
\dual{S_{\omega}''(\vec \phi_{\omega}) \vec v}{\vec v}
&=\dual{L_3 \reP v_1}{\reP v_1}+\dual{L_1 \imP v_1}{\imP v_1} \label{QF1} \\
&\hspace{5mm} 
+\dual{L_{\gamma/\kappa_1} \reP v_2}{\reP v_2}
+\dual{L_{-\gamma/\kappa_1} \imP v_2}{\imP v_2}. \nonumber 
\end{align}
We recall some known results on $L_a$ (see \cite{wei1}). 

\begin{lemma}\label{lem:pos}
\noindent ${\bf ({\rm 1})}$ \hspace{1mm}
If $1\le a\le 3$, then there exists $C>0$ such that 
$\dual{L_a v}{v}\ge C \|v\|_{H^1}^2$ 
for all $v\in H^1_{\even}(\R,\R)$ satisfying $(v,\varphi_{\omega})_{L^2}=0$. 

\noindent ${\bf ({\rm 2})}$ \hspace{1mm}
If $a<1$, then there exists $C>0$ such that $\dual{L_a v}{v}\ge C \|v\|_{H^1}^2$ 
for all $v\in H^1_{\even}(\R,\R)$. 

\noindent ${\bf ({\rm 3})}$ \hspace{1mm}
$L_1 \varphi_{\omega}=0$. 
If $a>1$, then $\dual{L_a \varphi_{\omega}}{\varphi_{\omega}}<0$. 
\end{lemma}

To prove the stability part of Theorem \ref{thm1}, 
we use the following proposition (see \cite{GSS1,wei2}). 

\begin{proposition}\label{prop2}
Assume that there exists a constant $C>0$ such that \par \noindent 
$\dual{S_{\omega}''(\vec \phi_{\omega})\vec w}{\vec w}\ge C \|\vec w\|_{X}^2$ 
for all $\vec w\in X$ satisfying 
\begin{equation} \label{OC1}
(\vec w, \vec \phi_{\omega})_{H}=(\vec w, i\vec \phi_{\omega})_{H}=0.
\end{equation}
Then, the standing wave solution $e^{i\omega t} \vec \phi_{\omega}$ of \eqref{nls} is stable. 
\end{proposition}

\begin{proof}[Proof of Theorem \ref{thm1} (Stability part)]
Assume that $\gamma<\kappa_1$. 

By Lemma \ref{lem:pos} (1), 
there exists $C_1>0$ such that 
$$\dual{L_3 \reP w_1}{\reP w_1}+\dual{L_1 \imP w_1}{\imP w_1}
\ge C_1 \|w_1\|_{H^1}^2$$
for all $w_1\in H^1(\R,\C)$ satisfying 
\begin{equation} \label{OC2}
(\reP w_1,\varphi_{\omega})_{L^2}=(\imP w_1, \varphi_{\omega})_{L^2}=0.
\end{equation} 
Note that since $\vec \phi_{\omega}$ has the form \eqref{ST1}, 
the condition \eqref{OC2} is equivalent to \eqref{OC1}. 
Moreover, by the assumption $0<\gamma<\kappa_1$, we have 
$-\gamma/\kappa_1<\gamma/\kappa_1<1$. 
Thus, by Lemma \ref{lem:pos} (2), 
there exists $C_2>0$ such that 
$$\dual{L_{\gamma/\kappa_1} \reP w_2}{\reP w_2}
+\dual{L_{-\gamma/\kappa_1} \imP w_2}{\imP w_2}
\ge C_2 \|w_2\|_{H^1}^2$$
for all $w_2\in H^1_{\even} (\R,\C)$. 

Thus, putting $C_3=\min \{C_1,C_2\}$, we have 
$\dual{S_{\omega}''(\vec \phi_{\omega})\vec w}{\vec w}\ge C_3 \|\vec w\|_{X}^2$ 
for all $\vec w\in X$ satisfying \eqref{OC1}. 
Hence, the stability part of Theorem \ref{thm1} follows from Proposition \ref{prop2}. 
\end{proof}

Next, we consider the instability part of Theorem \ref{thm1}. 
The instability of $e^{i\omega t}\vec \phi_{\omega}$ 
can be proved for all $\gamma\in (\kappa_1,\infty)$ 
in the same way as in \cite{CCO1,CCO2} using the linear instability argument. 
On the other hand, by the Lyapunov function method, 
the instability of $e^{i\omega t}\vec \phi_{\omega}$ is proved 
for a restricted case $\gamma\in (\kappa_1,3\kappa_1]$. 
Since our main interest in this paper is to consider the borderline case $\gamma=\kappa_1$ 
in Theorem \ref{thm2}, 
and since the instability result in Theorem \ref{thm2} is proved by the Lyapunov function method 
but not by the linear instability argument, 
we here give the proof of instability for the case $\gamma\in (\kappa_1,3\kappa_1]$. 
To prove the instability of $e^{i\omega t}\vec \phi_{\omega}$ for this case, 
we use the following proposition (see \cite{oht11}). 

\begin{proposition}\label{prop3}
Assume that there exist $\vec \psi\in X$ 
and a constant $C>0$ such that 
$$(\vec \psi,\vec \phi_{\omega})_{H}=(\vec \psi,i\vec \phi_{\omega})_{H}=0, \quad 
\dual{S_{\omega}''(\vec \phi_{\omega})\vec \psi}{\vec \psi}<0,$$ 
and 
$\dual{S_{\omega}''(\vec \phi_{\omega})\vec w}{\vec w}\ge C \|\vec w\|_{X}^2$ 
for all $\vec w\in  X$ satisfying 
\begin{equation} \label{OC3}
(\vec w, \vec \phi_{\omega})_{H}=(\vec w, i \vec \phi_{\omega})_{H}
=(\vec w, \vec \psi)_{H}=0.
\end{equation} 
Then, the standing wave solution $e^{i\omega t}\vec \phi_{\omega}$ of \eqref{nls} is unstable. 
\end{proposition}

\begin{proof}[Proof of Theorem \ref{thm1} (Instability part for the case 
$\kappa_1<\gamma\le 3\kappa_1$)] \hspace{1mm}

Assume that $\gamma\in (\kappa_1,3\kappa_1]$. 
We take
\begin{equation*}
\vec \psi_{\omega}=\left(0,\frac{1}{\sqrt{\kappa_1}} \varphi_{\omega}\right). 
\end{equation*}
Then, $\vec \psi_{\omega} \in X$ and $(\vec \psi_{\omega},\vec \phi_{\omega})_{H}
=(\vec \psi_{\omega},i\vec \phi_{\omega})_{H}=0$. 

Since $1<\gamma/\kappa_1\le 3$, 
by Lemma \ref{lem:pos} (3), we have 
$$\dual{S_{\omega}''(\vec \phi_{\omega})\vec \psi_{\omega}}{\vec \psi_{\omega}}
=\frac{1}{\kappa_1} \dual{L_{\gamma/\kappa_1} \varphi_{\omega}}{\varphi_{\omega}}<0.$$
Moreover, since the condition \eqref{OC3} is equivalent to 
\begin{equation*}
(\reP w_1,\varphi_{\omega})_{L^2}=(\imP w_1, \varphi_{\omega})_{L^2}
=(\reP w_2, \varphi_{\omega})_{L^2}=0, 
\end{equation*} 
by Lemma \ref{lem:pos} (1) and (2), 
there exists a constant $C>0$ such that \par \noindent 
$\dual{S_{\omega}''(\vec \phi_{\omega})\vec w}{\vec w}\ge C \|\vec w\|_{X}^2$ 
for all $\vec w\in  X$ satisfying \eqref{OC3}. 

Hence, the instability of $e^{i\omega t}\vec \phi_{\omega}$ follows from 
Proposition \ref{prop3}. 
\end{proof}

\section{Proof of Theorem \ref{thm2}} \label{sect3} 

In this section, we consider the case $\gamma=\kappa_1$. 
By \eqref{QF1}, we have 
\begin{align} \label{QF2}
\dual{S_{\omega}''(\vec \phi_{\omega}) \vec v}{\vec v}
&=\dual{L_3 \reP v_1}{\reP v_1}+\dual{L_1 \imP v_1}{\imP v_1} \\ 
&\hspace{5mm} 
+\dual{L_1 \reP v_2}{\reP v_2}
+\dual{L_{-1} \imP v_2}{\imP v_2} 
\nonumber
\end{align}
for $\vec v=(v_1,v_2)\in X$. 
Recall that 
\begin{equation*}
\vec \phi_{\omega}=\left(\frac{1}{\sqrt{\kappa_1}} \varphi_{\omega},0\right), \quad 
\vec \psi_{\omega}=\left(0,\frac{1}{\sqrt{\kappa_1}} \varphi_{\omega}\right). 
\end{equation*}
Then, we have 
\begin{align}
&\|\vec \psi_{\omega}\|_H=\|\vec \phi_{\omega}\|_H, \quad 
(\vec \psi_{\omega}, \vec \phi_{\omega})_H=(\vec \psi_{\omega},i\vec \phi_{\omega})_H=0, \nonumber \\
&S_{\omega}''(\vec \phi_{\omega})\vec \psi_{\omega}
=\left(0, \frac{1}{\sqrt{\kappa_1}} L_1 \varphi_{\omega} \right)
=\left(0, 0\right), \nonumber \\
&S_{\omega}''(\vec \phi_{\omega})\vec \phi_{\omega}
=S_{\omega}^{(3)}(\vec \phi_{\omega}) (\vec \psi_{\omega},\vec \psi_{\omega})
=\left(-\frac{2}{\sqrt{\kappa_1}} \varphi_{\omega}^3, 0 \right). \label{S2S3}
\end{align}
In particular, we have 
$$\dual{S_{\omega}^{(3)}(\vec \phi_{\omega}) (\vec \psi_{\omega},\vec \psi_{\omega})}
{\vec \psi_{\omega}}=0.$$

Moreover, we put 
\begin{align*}
&\nu_1:=\dual{S_{\omega}^{(4)} (\vec \phi_{\omega}) 
(\vec \psi_{\omega},\vec \psi_{\omega},\vec \psi_{\omega})}{\vec \psi_{\omega}}, \\
&\nu_0:=
\dfrac{1}{8} \dual{S_{\omega}''(\vec \phi_{\omega})\vec \phi_{\omega}}{\vec \phi_{\omega}} 
-\dfrac{1}{4} \dual{S_{\omega}^{(3)}(\vec \phi_{\omega}) (\vec \psi_{\omega},\vec \psi_{\omega})}
{\vec \phi_{\omega}}
+\frac{1}{4!}\nu_1.
\end{align*}
Then, by simple computations, we have 
\begin{equation}\label{nu0}
\nu_1=-\frac{6\kappa_2}{\kappa_1^2} \|\varphi_{\omega}\|_{L^4}^4, \quad 
\nu_0=\frac{\kappa_1-\kappa_2}{4\kappa_1^2} \|\varphi_{\omega}\|_{L^4}^4. 
\end{equation}
As we will see below, 
the sign of $\nu_0$ determines the stability and instability of $e^{i\omega t} \vec \phi_{\omega}$ 
for the borderline case $\gamma=\kappa_1$. 

The following lemma plays an important role in the proof of Theorem \ref{thm2}
for both stability and instability results. 

\begin{lemma} \label{lem:pW}
There exists a constant $k_0>0$ such that 
$$\dual{S_{\omega}''(\phi_{\omega})\vec w}{\vec w}\ge k_0 \|\vec w\|_X^2$$ 
for all $\vec w\in W$, where 
$$W=\{\vec w\in X: (\vec w, \vec \phi_{\omega})_H=(\vec w, i\vec \phi_{\omega})_H
=(\vec w, \vec \psi_{\omega})_H=0\}.$$
\end{lemma}

\begin{proof}
Since $\vec w\in W$ satisfies 
\begin{equation*}
(\reP w_1,\varphi_{\omega})_{L^2}=(\imP w_1, \varphi_{\omega})_{L^2}
=(\reP w_2, \varphi_{\omega})_{L^2}=0, 
\end{equation*} 
the conclusion follows from \eqref{QF2} and Lemma \ref{lem:pos}. 
\end{proof}

\begin{lemma} \label{lem:A3}
For $\lambda\in \R$, 
\begin{align*}
S_{\omega}(\vec \phi_{\omega}+\lambda \vec \psi_{\omega})
=S_{\omega}(\vec \phi_{\omega})+\frac{\nu_1}{4!}\lambda^4, \quad 
\dual{S_{\omega}'(\vec \phi_{\omega}+\lambda \vec \psi_{\omega})}{\vec \psi_{\omega}}
=\frac{\nu_1}{3!} \lambda^3. 
\end{align*}
\end{lemma}

\begin{proof}
By Taylor's expansion, we have 
\begin{align*}
&S_{\omega}(\vec \phi_{\omega}+\lambda \vec \psi_{\omega})
=S_{\omega}(\vec \phi_{\omega})
+\lambda  \dual{S_{\omega}'(\vec \phi_{\omega})}{\vec \psi_{\omega}} 
+\frac{\lambda^2}{2} \dual{S_{\omega}''(\vec \phi_{\omega}) \vec \psi_{\omega}}{ \vec \psi_{\omega}} \\
&\hspace{10mm}
+\frac{\lambda^3}{3!} \dual{S_{\omega}^{(3)}(\vec \phi_{\omega})(\vec \psi_{\omega},\vec \psi_{\omega})}
{ \vec \psi_{\omega}} 
+\frac{\lambda^4}{4!} 
\dual{S_{\omega}^{(4)}(\vec \phi_{\omega})(\vec \psi_{\omega},\vec \psi_{\omega},\vec \psi_{\omega})}
{ \vec \psi_{\omega}}. 
\end{align*}
Since $S_{\omega}'(\vec \phi_{\omega})=S_{\omega}''(\vec \phi_{\omega}) \vec \psi_{\omega}=0$
and $\dual{S_{\omega}^{(3)}(\vec \phi_{\omega}) (\vec \psi_{\omega},\vec \psi_{\omega})}
{\vec \psi_{\omega}}=0$, we have 
\begin{align*}
&S_{\omega}(\vec \phi_{\omega}+\lambda \vec \psi_{\omega})
=S_{\omega}(\vec \phi_{\omega})+\frac{\nu_1}{4!} \lambda^4. 
\end{align*}

Moreover, by differentiating this identity with respect to $\lambda$, 
we have the second identity. 
\end{proof}

\begin{lemma} \label{lem:S0} 
For $\lambda\in \R$, 
\begin{align*}
S_{\omega}'(\vec \phi_{\omega}+\lambda \vec \psi_{\omega})
=\frac{\lambda^2}{2} S_{\omega}^{(3)}(\vec \phi_{\omega}) (\vec \psi_{\omega},\vec \psi_{\omega})
+o(\lambda^2). 
\end{align*}
\end{lemma}

\begin{proof}
Since $S_{\omega}'(\vec \phi_{\omega})=S_{\omega}''(\vec \phi_{\omega}) \vec \psi_{\omega}=0$, 
we have 
\begin{align*}
&S_{\omega}'(\vec \phi_{\omega}+\lambda \vec \psi_{\omega}) \\
&=S_{\omega}'(\vec \phi_{\omega})
+\lambda S_{\omega}''(\vec \phi_{\omega}) \vec \psi_{\omega} 
+\frac{\lambda^2}{2} S_{\omega}^{(3)}(\vec \phi_{\omega}) (\vec \psi_{\omega},\vec \psi_{\omega})
+o(\lambda^2) \\
&=\frac{\lambda^2}{2} S_{\omega}^{(3)}(\vec \phi_{\omega}) (\vec \psi_{\omega},\vec \psi_{\omega})
+o(\lambda^2). 
\end{align*}
This completes the proof. 
\end{proof}

\begin{lemma}\label{lem:S1}
For $\lambda\in \R$ and $\vec z\in X$, 
\begin{align*}
&S_{\omega}(\vec \phi_{\omega}+\lambda \vec \psi_{\omega}+\vec z)
-S_{\omega}(\vec \phi_{\omega}) \\
&=\frac{\nu_1}{4!} \lambda^4
+\frac{\lambda^2}{2} \dual{S_{\omega}^{(3)}(\vec \phi_{\omega})(\vec \psi_{\omega},\vec \psi_{\omega})}
{\vec z} 
+\frac{1}{2}\dual{S_{\omega}''(\vec \phi_{\omega}) \vec z}{\vec z} \\
&\hspace{7mm} 
+o(\lambda^4+\|\vec z\|_X^2). 
\end{align*}
\end{lemma}

\begin{proof}
By Taylor's expansion, we have 
\begin{align*}
&S_{\omega}(\vec \phi_{\omega}+\lambda \vec \psi_{\omega}+\vec z)
=S_{\omega}(\vec \phi_{\omega}+\lambda \vec \psi_{\omega}) \\
&\hspace{10mm} 
+\dual{S_{\omega}'(\vec \phi_{\omega}+\lambda \vec \psi_{\omega})}{\vec z}
+\frac{1}{2} \dual{S_{\omega}''(\vec \phi_{\omega}+\lambda \vec \psi_{\omega}) \vec z}{\vec z} 
+o(\|\vec z\|_X^2). 
\end{align*}
Here, by Lemma \ref{lem:A3}, we have 
$S_{\omega}(\vec \phi_{\omega}+\lambda \vec \psi_{\omega})
=S_{\omega}(\vec \phi_{\omega})+\dfrac{\nu_1}{4!} \lambda^4$. 

Next, it follows from Lemma \ref{lem:S0} that 
\begin{align*}
\dual{S_{\omega}'(\vec \phi_{\omega}+\lambda \vec \psi_{\omega})}{\vec z} 
=\frac{\lambda^2}{2} \dual{S_{\omega}^{(3)}(\vec \phi_{\omega})(\vec \psi_{\omega},\vec \psi_{\omega})}
{\vec z}+o(\lambda^4+\|\vec z\|_X^2). 
\end{align*}
Moreover, we have 
$$\dual{S_{\omega}''(\vec \phi_{\omega}+\lambda \vec \psi_{\omega}) \vec z}{\vec z} 
=\dual{S_{\omega}''(\vec \phi_{\omega}) \vec z}{\vec z}
+O(\lambda \|\vec z\|_X^2).$$

Thus, we have the desired estimate. 
\end{proof}

\begin{lemma}\label{lem:Q1}
Let $\vec v=\lambda \vec \psi_{\omega}+\mu \vec \phi_{\omega}+\vec w$ 
with $\lambda$, $\mu\in \R$ and $\vec w\in W$. 
Assume that $\|\vec \phi_{\omega}+\vec v\|_H^2=\|\vec \phi_{\omega}\|_H^2$. 
Then, 
$$\mu=-\frac{\lambda^2}{2}+O(\mu^2+\|\vec w\|_X^2).$$
\end{lemma}

\begin{proof}
Since $\vec \psi_{\omega}$, $\vec \phi_{\omega}$ and $\vec w$ are orthogonal to each other in $H$, 
we have 
$$\|\vec \phi_{\omega}\|_H^2=\|\vec \phi_{\omega}+\vec v\|_H^2
=\lambda^2 \|\vec \psi_{\omega}\|_H^2+(1+\mu)^2 \|\vec \phi_{\omega}\|_H^2+\|\vec w\|_H^2.$$
Moreover, since $\|\vec \psi_{\omega}\|_H=\|\vec \phi_{\omega}\|_H$, we have 
$$\mu=-\frac{\lambda^2}{2}
+\frac{1}{2}\left(\mu^2+\frac{\|\vec w\|_H^2}{\|\vec \phi_{\omega}\|_H^2}\right),$$
which implies the desired result. 
\end{proof}

\begin{lemma}\label{lem:QS1}
Let $\vec v=\lambda \vec \psi_{\omega}+\mu \vec \phi_{\omega}+\vec w$ 
with $\lambda$, $\mu\in \R$ and $\vec w\in W$. 
Assume that $\|\vec \phi_{\omega}+\vec v\|_H^2=\|\vec \phi_{\omega}\|_H^2$.
Then, 
\begin{align*}
E(\vec \phi_{\omega}+\vec v)-E(\vec \phi_{\omega}) 
&=\nu_0 \lambda^4
+\frac{1}{2}\dual{S_{\omega}''(\vec \phi_{\omega}) \vec w}{\vec w} 
+o(\lambda^4+\|\vec w\|_X^2).
\end{align*}
\end{lemma}

\begin{proof}
By Lemmas \ref{lem:S1} and \ref{lem:Q1}, we have 
\begin{align*}
&E(\vec \phi_{\omega}+\vec v)-E(\vec \phi_{\omega}) 
=S_{\omega}(\vec \phi_{\omega}+\vec v)-S_{\omega}(\vec \phi_{\omega}) \\
&=S_{\omega}(\vec \phi_{\omega}+\lambda \vec \psi_{\omega}+\mu \vec \phi_{\omega}+\vec w)
-S_{\omega}(\vec \phi_{\omega}) \\
&=\frac{\nu_1}{4!} \lambda^4
+\frac{\lambda^2}{2} 
\dual{S_{\omega}^{(3)}(\vec \phi_{\omega})(\vec \psi_{\omega},\vec \psi_{\omega})}
{\mu \vec \phi_{\omega}+\vec w} \\
&\hspace{3mm}
+\frac{1}{2}\dual{S_{\omega}''(\vec \phi_{\omega})(\mu \vec \phi_{\omega}+\vec w)}
{\mu \vec \phi_{\omega}+\vec w} 
+o(\lambda^4+\|\mu \vec \phi_{\omega}+\vec w\|_X^2) \\
&=\nu_0 \lambda^4
+\frac{\lambda^2}{2} 
\dual{S_{\omega}^{(3)}(\vec \phi_{\omega})(\vec \psi_{\omega},\vec \psi_{\omega})}{\vec w} 
-\frac{\lambda^2}{2}\dual{S_{\omega}''(\vec \phi_{\omega})\vec \phi_{\omega}}{\vec w} \\
&\hspace{3mm} 
+\frac{1}{2}\dual{S_{\omega}''(\vec \phi_{\omega}) \vec w}{\vec w} 
+o(\lambda^4+\|\vec w\|_X^2). 
\end{align*}
Here, by \eqref{S2S3}, 
the second and the third terms in the last equation 
cancel each other out. 
This completes the proof. 
\end{proof}

To prove the stability part of Theorem \ref{thm2}, 
we use the following proposition (see \cite{GSS1}). 
For $\eps>0$, we define 
$$U_{\eps}(\vec \phi_{\omega})
=\{\vec u\in X: \inf_{\theta\in \R} \|\vec u-e^{i\theta} \vec \phi_{\omega}\|_X<\eps\}.$$

\begin{proposition}\label{prop4}
Assume that there exist positive constants 
$p$, $C$ and $\eps$ such that 
\begin{align*}
E(\vec u) \ge E(\vec \phi_{\omega}) 
+C \inf_{\theta\in \R} \|\vec u-e^{i\theta} \vec \phi_{\omega}\|_X^p 
\end{align*}
for all $\vec u\in U_{\eps}(\vec \phi_{\omega})$ satisfying 
$Q(\vec u)=Q(\vec \phi_{\omega})$. 
Then, the standing wave solution $e^{i\omega t} \vec \phi_{\omega}$ of \eqref{nls} is stable. 
\end{proposition}

Before proving the stability part of Theorem \ref{thm2}, 
we prepare one more lemma. 

\begin{lemma}\label{lem:ift}
There exist $\varepsilon>0$ and a $C^2$-function 
$\alpha: U_{\eps}(\vec \phi_{\omega})\to \R/2\pi \Z$ such that 
\begin{align} 
&\|\vec u-e^{i\alpha (\vec u)}\vec \phi_{\omega}\|_H
\le \|\vec u-e^{i\theta} \vec \phi_{\omega}\|_H, \quad 
\alpha (e^{i\theta} \vec u)=\alpha (\vec u)+\theta, 
\nonumber \\
&(\vec u, e^{i\alpha (\vec u)} i \vec \phi_{\omega})_H=0, \quad 
i \alpha'(\vec u)=-\frac{e^{i\alpha (\vec u)} \vec \phi_{\omega}}
{(\vec u, e^{i\alpha (\vec u)} \vec \phi_{\omega})_H}
\label{i-alpha}
\end{align}
for all $\vec u\in U_{\varepsilon}(\vec \phi_{\omega})$ and $\theta\in \R/2\pi \Z$. 
\end{lemma}

\begin{proof}
See Lemma 3.2 of \cite{GSS1}. 
\end{proof}

\begin{proof}[Proof of Theorem \ref{thm2} (Stability part)]
Assume that $\gamma=\kappa_1>\kappa_2$. 

Let $\vec u\in U_{\eps}(\vec \phi_{\omega})$ with $Q(\vec u)=Q(\vec \phi_{\omega})$. 
Then, for $\alpha(\vec u)$ given in Lemma \ref{lem:ift}, we have 
\begin{equation} \label{alpha-eps}
\|\vec u-e^{i\alpha (\vec u)}\vec \phi_{\omega}\|_X 
\le \left(1+\frac{2 \|\vec \phi_{\omega}\|_X}{\|\vec \phi_{\omega}\|_H}\right)\eps. 
\end{equation}
Indeed, for $\beta (\vec u)\in \R$ such that 
\begin{align*} 
\|\vec u-e^{i\beta (\vec u)}\vec \phi_{\omega}\|_X
=\inf_{\theta\in \R} \|\vec u-e^{i\theta} \vec \phi_{\omega}\|_X<\eps, 
\end{align*}
we have 
\begin{align*}
&|e^{i\alpha (\vec u)}-e^{i\beta (\vec u)}|\|\vec \phi_{\omega}\|_H
\le \|e^{i\alpha (\vec u)}\vec \phi_{\omega}-\vec u\|_H
+\|\vec u-e^{i\beta (\vec u)}\vec \phi_{\omega}\|_H \\
&\le 2 \|\vec u-e^{i\beta (\vec u)}\vec \phi_{\omega}\|_H
\le 2 \|\vec u-e^{i\beta (\vec u)}\vec \phi_{\omega}\|_X<2\eps,
\end{align*}
and 
$\|\vec u-e^{i\alpha (\vec u)}\vec \phi_{\omega}\|_X 
\le \|\vec u-e^{i\beta (\vec u)}\vec \phi_{\omega}\|_X
+|e^{i\alpha (\vec u)}-e^{i\beta (\vec u)}|\|\vec \phi_{\omega}\|_X$, 
which implies \eqref{alpha-eps}. 

Let $\vec v=e^{-i\alpha (\vec u)}\vec u-\vec \phi_{\omega}$. 
Then, we have $(\vec v, i\vec \phi_{\omega})_H=0$, 
and we decompose $\vec v$ as 
$\vec v=\lambda \vec \psi_{\omega}+\mu \vec \phi_{\omega}+\vec w$ 
with $\lambda$, $\mu\in \R$ and $\vec w\in W$. 

Since $\|\vec \phi_{\omega}+\vec v\|_H^2=\|\vec u\|_H^2
=2Q(\vec u)=2Q(\vec \phi_{\omega})=\|\vec \phi_{\omega}\|_H^2$, 
it follows from Lemmas \ref{lem:QS1} and \ref{lem:pW} that 
\begin{align*}
E(\vec u)-E(\vec \phi_{\omega}) 
\ge \nu_0 \lambda^4+\frac{k_0}{2} \|\vec w\|_X^2
+o(\lambda^4+\|\vec w\|_X^2). 
\end{align*}
Here, we note that $k_0$ is the positive constant given in Lemma \ref{lem:pW}, 
and that $\nu_0>0$ 
by \eqref{nu0} and the assumption $\kappa_1>\kappa_2$. 

Moreover, by Lemma \ref{lem:Q1}, we have 
\begin{align*}
&\inf_{\theta\in \R} \|\vec u-e^{i\theta} \vec \phi_{\omega}\|_X
\le \|\vec v\|_X
\le |\lambda| \|\vec \psi_{\omega}\|_X+|\mu| \|\vec \phi_{\omega}\|_X+\|\vec w\|_X \\
&=|\lambda| \|\vec \psi_{\omega}\|_X+\|\vec w\|_X+O(\lambda^2+\|\vec w\|_X^2). 
\end{align*}

Thus, taking $\eps$ smaller if necessary, we have 
\begin{align*}
E(\vec u)-E(\vec \phi_{\omega}) 
\ge \frac{\nu_0}{2} \lambda^4+\frac{k_0}{4} \|\vec w\|_X^2
\ge C_1 \inf_{\theta\in \R} \|\vec u-e^{i\theta} \vec \phi_{\omega}\|_X^4
\end{align*}
for some $C_1>0$. 

Hence, the stability of $e^{i\omega t} \vec \phi_{\omega}$ 
follows from Proposition \ref{prop4}. 
\end{proof}

In the rest of this section, 
we study the instability of $e^{i\omega t} \vec \phi_{\omega}$ 
for the case $\gamma=\kappa_1<\kappa_2$. 
We follow the argument used in \cite{CO3, oht11}. 

For $\vec u\in U_{\eps}(\vec \phi_{\omega})$, we define 
\begin{align}
&M(\vec u)=e^{-i\alpha (\vec u)} \vec u, \quad 
A(\vec u)=(i M(\vec u), \vec \psi_{\omega})_H, \label{eq:MA} \\
&q(\vec u)=e^{i\alpha (\vec u)} \vec \psi_{\omega}
+(M(\vec u), \vec \psi_{\omega})_H\, i\alpha'(\vec u), \nonumber \\
&P(\vec u)=\dual{E'(\vec u)}{q(\vec u)}. \nonumber 
\end{align}

Then, we have the following lemmas (see \cite{GSS1}). 

\begin{lemma}\label{lem:Aq}
For $\vec u\in U_{\eps} (\vec \phi_{\omega})$, 

\noindent ${\bf ({\rm 1})}$ \hspace{1mm}
$A(e^{i\theta} \vec u)=A(\vec u)$, \hspace{1mm} 
$q(e^{i\theta} \vec u)=e^{i\theta} q(\vec u)$ 
for all $\theta \in \R$. 

\noindent ${\bf ({\rm 2})}$ \hspace{1mm}
$\dual{A'(\vec u)}{\vec w}=(q(\vec u),i \vec w)_H$ for $\vec w\in X$. 

\noindent ${\bf ({\rm 3})}$ \hspace{1mm}
$q(\vec \phi_{\omega})=\vec \psi_{\omega}$, \hspace{1mm} 
$\dual{Q'(\vec u)}{q(\vec u)}=0$. 
\end{lemma}

\begin{lemma}\label{lem:AP}
Let $I$ be an interval of $\R$. 
Let $\vec u\in C(I,X)$ be a solution of \eqref{nls}, 
and assume that $\vec u(t)\in U_{\varepsilon}(\vec \phi_{\omega})$ for all $t\in I$. 
Then,
$$\frac{d}{dt}A(\vec u(t))=P(\vec u(t))
\quad \mbox{for all} \hspace{2mm} t\in I.$$
\end{lemma}

By Lemma \ref{lem:Aq} and \eqref{i-alpha}, we have 
\begin{align}
&P(\vec u)=\dual{S_{\omega}'(\vec u)}{q(\vec u)} 
\nonumber \\
&=\dual{S_{\omega}' \left(M(\vec u)\right)}{\vec \psi_{\omega}}
-\frac{(M(\vec u), \vec \psi_{\omega})_H}{(M(\vec u), \vec \phi_{\omega})_H}
\dual{S_{\omega}' \left(M(\vec u)\right)}{\vec \phi_{\omega}}. 
\label{repP}
\end{align}

We prove the following. 

\begin{proposition}\label{prop6}
Let $\gamma=\kappa_1<\kappa_2$. 
Then, there exists a constant $\eps_0>0$ such that 
\begin{align*}
E(\vec \phi_{\omega}) 
\le E(\vec u)-\frac{(M(\vec u),\vec \psi_{\omega})_H}{2 \|\vec \psi_{\omega}\|_H^2} P(\vec u)
\end{align*}
for all $\vec u\in U_{\eps_0}(\vec \phi_{\omega})$ satisfying 
$Q(\vec u)=Q(\vec \phi_{\omega})$. 
\end{proposition}

For the proof of Proposition \ref{prop6}, 
we prove several lemmas. 

\begin{lemma}\label{lem:S21}
For $\lambda\in \R$ and $\vec z\in X$, 
\begin{align*}
&\lambda \dual{S_{\omega}'(\vec \phi_{\omega}+\lambda \vec \psi_{\omega}+\vec z)}
{\vec \psi_{\omega}} \\
&=\dfrac{\nu_1}{3!} \lambda^4
+\lambda^2 \dual{S_{\omega}^{(3)}(\vec \phi_{\omega})(\vec \psi_{\omega},\vec \psi_{\omega})}{\vec z}
+o(\lambda^4+\|\vec z\|_X^2). 
\end{align*}
\end{lemma}

\begin{proof}
By Taylor's expansion, we have 
\begin{align*}
&\lambda \dual{S_{\omega}'(\vec \phi_{\omega}+\lambda \vec \psi_{\omega}
+\vec z)}{\vec \psi_{\omega}} \\
&=\lambda \dual{S_{\omega}'(\vec \phi_{\omega}+\lambda \vec \psi_{\omega})}{\vec \psi_{\omega}}
+\lambda \dual{S_{\omega}''(\vec \phi_{\omega}+\lambda \vec \psi_{\omega}) \vec z}{\vec \psi_{\omega}}
+O(\lambda \|\vec z\|_X^2). 
\end{align*}
Here, by Lemma \ref{lem:A3}, we have 
$\lambda \dual{S_{\omega}'(\vec \phi_{\omega}+\lambda \vec \psi_{\omega})}{\vec \psi_{\omega}}
=\dfrac{\nu_1}{3!} \lambda^4$. 

Next, since $S_{\omega}''(\vec \phi_{\omega})\vec \psi_{\omega}=0$, we have 
\begin{align*}
&\lambda \dual{S_{\omega}''(\vec \phi_{\omega}+\lambda \vec \psi_{\omega}) \vec z}
{\vec \psi_{\omega}} \\
&=\lambda \dual{S_{\omega}''(\vec \phi_{\omega}) \vec z}{\vec \psi_{\omega}}
+\lambda^2 \dual{S_{\omega}^{(3)}(\vec \phi_{\omega})(\vec \psi_{\omega}, \vec z)}{\vec \psi_{\omega}}
+o(\lambda^2 \|\vec z\|_X) \\
&=\lambda^2 \dual{S_{\omega}^{(3)}(\vec \phi_{\omega})(\vec \psi_{\omega},\vec \psi_{\omega})}{\vec z}
+o(\lambda^2 \|\vec z\|_X). 
\end{align*}

Thus, we obtain the desired result. 
\end{proof}

\begin{lemma}\label{lem:S22}
For $\lambda\in \R$ and $\vec z\in X$, 
\begin{align*}
&\lambda^2 \dual{S_{\omega}'(\vec \phi_{\omega}+\lambda \vec \psi_{\omega}+\vec z)}
{\vec \phi_{\omega}} \\
&=\frac{\lambda^4}{2} \dual{S_{\omega}^{(3)}(\vec \phi_{\omega}) 
(\vec \psi_{\omega},\vec \psi_{\omega})}{\vec \phi_{\omega}}
+\lambda^2 \dual{S_{\omega}''(\vec \phi_{\omega})\vec \phi_{\omega}}{\vec z} 
+o(\lambda^4+\|\vec z\|_X^2). 
\end{align*}
\end{lemma}

\begin{proof}
By Taylor's expansion, we have 
\begin{align*}
&\dual{S_{\omega}'(\vec \phi_{\omega}+\lambda \vec \psi_{\omega}+\vec z)}
{\vec \phi_{\omega}} \\
&=\dual{S_{\omega}'(\vec \phi_{\omega}+\lambda \vec \psi_{\omega})}
{\vec \phi_{\omega}}
+\dual{S_{\omega}''(\vec \phi_{\omega}+\lambda \vec \psi_{\omega}) \vec z}
{\vec \phi_{\omega}}
+O(\|\vec z\|_X^2). 
\end{align*}
Here, it follows from Lemma \ref{lem:S0} that 
\begin{align*}
\dual{S_{\omega}'(\vec \phi_{\omega}+\lambda \vec \psi_{\omega})}
{\vec \phi_{\omega}}
=\frac{\lambda^2}{2} \dual{S_{\omega}^{(3)}(\vec \phi_{\omega}) 
(\vec \psi_{\omega},\vec \psi_{\omega})}{\vec \phi_{\omega}}
+o(\lambda^2). 
\end{align*}
Moreover, we have 
$\dual{S_{\omega}''(\vec \phi_{\omega}+\lambda \vec \psi_{\omega}) \vec z}{\vec \phi_{\omega}}
=\dual{S_{\omega}''(\vec \phi_{\omega}) \vec z}{\vec \phi_{\omega}}
+O(\lambda \|\vec z\|_X)$. 

Thus, we have 
\begin{align*}
&\dual{S_{\omega}'(\vec \phi_{\omega}+\lambda \vec \psi_{\omega}+\vec z)}
{\vec \phi_{\omega}}
=\frac{\lambda^2}{2} \dual{S_{\omega}^{(3)}(\vec \phi_{\omega})
(\vec \psi_{\omega},\vec \psi_{\omega})}{\vec \phi_{\omega}} \\
&\hspace{20mm} 
+\dual{S_{\omega}''(\vec \phi_{\omega}) \vec \phi_{\omega}}{\vec z} 
+o(\lambda^2)+O(\lambda \|\vec z\|_X)+O(\|\vec z\|_X^2), 
\end{align*}
which implies the desired result. 
\end{proof}

\begin{lemma}\label{lem:QS2}
Let $\vec v=\lambda \vec \psi_{\omega}
+\mu \vec \phi_{\omega}+\vec w$ with $\lambda$, $\mu\in \R$ and $\vec w\in W$. 
Assume that $\|\vec \phi_{\omega}+\vec v\|_H^2=\|\vec \phi_{\omega}\|_H^2$. 
Then, 
\begin{align*}
\lambda \dual{S_{\omega}'(\vec \phi_{\omega}+\vec v)}{\vec \psi_{\omega}}
-\lambda^2 \dual{S_{\omega}'(\vec \phi_{\omega}+\vec v)}{\vec \phi_{\omega}} 
=4 \nu_0 \lambda^4+o(\lambda^4+\|\vec w\|_X^2). 
\end{align*}
\end{lemma}

\begin{proof}
By Lemmas \ref{lem:Q1} and \ref{lem:S21}, we have 
\begin{align*}
&\lambda \dual{S_{\omega}'(\vec \phi_{\omega}+\vec v)}{\vec \psi_{\omega}} 
=\lambda \dual{S_{\omega}'(\vec \phi_{\omega}+\lambda \vec \psi_{\omega}
+\mu \vec \phi_{\omega}+\vec w)}{\vec \psi_{\omega}} \\
&=\dfrac{\nu_1}{3!} \lambda^4
-\frac{\lambda^4}{2} \dual{S_{\omega}^{(3)}(\vec \phi_{\omega})
(\vec \psi_{\omega},\vec \psi_{\omega})}{\vec \phi_{\omega}} \\
&\hspace{8mm} 
+\lambda^2 \dual{S_{\omega}^{(3)}(\vec \phi_{\omega})
(\vec \psi_{\omega},\vec \psi_{\omega})}{\vec w} 
+o(\lambda^4+\|\vec w\|_X^2). 
\end{align*}

On the other hand, 
by Lemmas \ref{lem:Q1} and \ref{lem:S22}, we have 
\begin{align*}
&\lambda^2 \dual{S_{\omega}'(\vec \phi_{\omega}+\vec v)}{\vec \phi_{\omega}}
=\lambda^2 \dual{S_{\omega}'(\vec \phi_{\omega}+\lambda \vec \psi_{\omega}
+\mu \vec \phi_{\omega}+\vec w)}{\vec \phi_{\omega}} \\
&=\frac{\lambda^4}{2} \dual{S_{\omega}^{(3)}(\vec \phi_{\omega}) 
(\vec \psi_{\omega},\vec \psi_{\omega})}{\vec \phi_{\omega}}
+\lambda^2 \mu \dual{S_{\omega}''(\vec \phi_{\omega}) \vec \phi_{\omega}}{\vec \phi_{\omega}} \\
&\hspace{10mm} 
+\lambda^2 \dual{S_{\omega}''(\vec \phi_{\omega}) \vec \phi_{\omega}}{\vec w} 
+o(\lambda^4+\|\mu \vec \phi_{\omega}+\vec w\|_X^2) \\
&=\frac{\lambda^4}{2} \dual{S_{\omega}^{(3)}(\vec \phi_{\omega}) 
(\vec \psi_{\omega},\vec \psi_{\omega})}{\vec \phi_{\omega}}
-\frac{\lambda^4}{2} \dual{S_{\omega}''(\vec \phi_{\omega}) \vec \phi_{\omega}}{\vec \phi_{\omega}} \\
&\hspace{10mm} 
+\lambda^2 \dual{S_{\omega}''(\vec \phi_{\omega}) \vec \phi_{\omega}}{\vec w}
+o(\lambda^4+\|\vec w\|_X^2). 
\end{align*}

Thus, we have 
\begin{align*}
&\lambda \dual{S_{\omega}'(\vec \phi_{\omega}+\vec v)}{\vec \psi_{\omega}}
-\lambda^2 \dual{S_{\omega}'(\vec \phi_{\omega}+\vec v)}{\vec \phi_{\omega}} \\
&=4 \nu_0 \lambda^4
-\lambda^2 \dual{S_{\omega}''(\vec \phi_{\omega})\vec \phi_{\omega}}{\vec w} 
+\lambda^2 \dual{S_{\omega}^{(3)}(\vec \phi_{\omega})
(\vec \psi_{\omega},\vec \psi_{\omega})}{\vec w} \\
&\hspace{5mm} 
+o(\lambda^4+\|\vec w\|_X^2). 
\end{align*}
Finally, by \eqref{S2S3}, we obtain the desired result. 
\end{proof}

We are now in a position to give the proof of Proposition \ref{prop6}. 

\begin{proof}[Proof of Proposition \ref{prop6}] 
Let $\vec u\in U_{\eps}(\vec \phi_{\omega})$ with $Q(\vec u)=Q(\vec \phi_{\omega})$. 
We put $\vec v=M(\vec u)-\vec \phi_{\omega}$, 
and decompose $\vec v$ as 
$\vec v=\lambda \vec \psi_{\omega}+\mu \vec \phi_{\omega}+\vec w$ 
with $\lambda$, $\mu\in \R$ and $\vec w\in W$. 
Here, we note that $(\vec v, i\vec \phi_{\omega})_H=0$ by Lemma \ref{lem:ift}, 
$$\lambda=\frac{(\vec v,\vec \psi_{\omega})_H}{\|\vec \psi_{\omega}\|_H^2}
=\frac{(M(\vec u),\vec \psi_{\omega})_H}{\|\vec \psi_{\omega}\|_H^2},$$
and $\|\vec v\|_X\le C\eps$ for some $C>0$  by \eqref{alpha-eps}. 

Since $\|\vec \phi_{\omega}+\vec v\|_H^2=\|\vec u\|_H^2
=2Q(\vec u)=2Q(\vec \phi_{\omega})=\|\vec \phi_{\omega}\|_H^2$, 
it follows from Lemmas \ref{lem:QS1} and \ref{lem:pW} that 
\begin{align*}
E(\vec u)-E(\vec \phi_{\omega}) 
\ge \nu_0 \lambda^4+\frac{k_0}{2} \|\vec w\|_X^2
+o(\lambda^4+\|\vec w\|_X^2). 
\end{align*}

Moreover, by \eqref{repP} and Lemmas \ref{lem:Q1}, \ref{lem:S22} and \ref{lem:QS2}, 
we have 
\begin{align*}
&\lambda P(\vec u)
=\lambda \dual{S_{\omega}' (\vec \phi_{\omega}+\vec v)}{\vec \psi_{\omega}}
-\lambda \frac{(\vec \phi_{\omega}+\vec v, \vec \psi_{\omega})_H}
{(\vec \phi_{\omega}+\vec v, \vec \phi_{\omega})_H}
\dual{S_{\omega}' (\vec \phi_{\omega}+\vec v)}{\vec \phi_{\omega}} \\
&=\lambda \dual{S_{\omega}' (\vec \phi_{\omega}+\vec v)}{\vec \psi_{\omega}}
-\frac{\lambda^2 \|\vec \psi_{\omega}\|_H^2}{(1+\mu)\|\vec \phi_{\omega}\|_H^2}
\dual{S_{\omega}' (\vec \phi_{\omega}+\vec v)}{\vec \phi_{\omega}} \\
&=\lambda \dual{S_{\omega}' (\vec \psi_{\omega}+\vec v)}{\vec \psi_{\omega}}
-\lambda^2 \dual{S_{\omega}' (\vec \phi_{\omega}+\vec v)}{\vec \phi_{\omega}} 
+o(\lambda^4+\|\vec w\|_X^2) \\
&=4 \nu_0 \lambda^4+o(\lambda^4+\|\vec w\|_X^2). 
\end{align*}
Here, we used the fact that $\|\vec \psi_{\omega}\|_H=\|\vec \phi_{\omega}\|_H$. 
Thus, we have 
\begin{align*}
E(\vec u)-E(\vec \phi_{\omega})-\frac{\lambda}{2} P(\vec u)
\ge -\nu_0 \lambda^4+\frac{k_0}{2} \|\vec w\|_X^2
+o(\lambda^4+\|\vec w\|_X^2). 
\end{align*}
Since $k_0>0$ and $-\nu_0>0$ 
by \eqref{nu0} and the assumption $\kappa_1<\kappa_2$, 
taking $\eps$ smaller if necessary, we have 
\begin{align*}
E(\vec u)-E(\vec \phi_{\omega})
\ge \frac{\lambda}{2} P(\vec u)
=\frac{(M(\vec u),\vec \psi_{\omega})_H}{2 \|\vec \psi_{\omega}\|_H^2} P(\vec u). 
\end{align*}
This completes the proof. 
\end{proof}

Finally, we prove the instability part of Theorem \ref{thm2}. 

\begin{proof}[Proof of Theorem \ref{thm2} (Instability part)] 
Suppose that $e^{i\omega t}\vec \phi_{\omega}$ is stable. 
For $\lambda$ close to $0$, we define 
$$\vec \varphi_{\lambda}=\vec \phi_{\omega}+\lambda \vec \psi_{\omega}
+\sigma(\lambda) \vec \phi_{\omega}, \quad 
\sigma(\lambda)=(1-\lambda^2)^{1/2}-1.$$
Then, we have $Q(\vec \varphi_{\lambda})=Q(\vec \phi_{\omega})$. 
Moreover, since $\nu_0<0$, 
by Lemma \ref{lem:QS1}, 
there exists $\lambda_1>0$ such that 
\begin{align*}
\delta_{\lambda}:=E(\vec \phi_{\omega})-E(\vec \varphi_{\lambda})
=-\nu_0 \lambda^4+o(\lambda^4)>0
\end{align*}
for $\lambda\in (-\lambda_1,0)\cup (0,\lambda_1)$. 

Let $\vec u_{\lambda}(t)$ be the solution of \eqref{nls} 
with $\vec u_{\lambda}(0)=\vec \varphi_{\lambda}$. 
Since $e^{i\omega t}\vec \phi_{\omega}$ is stable, 
there exists $\lambda_0\in (0,\lambda_1)$ such that 
if $|\lambda|<\lambda_0$, 
then $\vec u_{\lambda}(t)\in U_{\eps_0}(\vec \phi_{\omega})$
for all $t\ge 0$, 
where $\eps_0$ is the positive conatant given in Proposition \ref{prop6}. 

By the definition \eqref{eq:MA} of $M$ and $A$, 
there exist positive constants $C_1$ and $C_2$ such that 
$$|M(\vec v)|\le C_1 \|\vec \psi_{\omega}\|_H, \quad 
|A(\vec v)|\le C_2$$
for all $\vec v\in U_{\eps_0}(\vec \phi_{\omega})$.  

For $\lambda\in (-\lambda_0,0)\cup (0,\lambda_0)$, 
by Proposition \ref{prop6} and the conservation of $E$ and $Q$, we have 
$$0<\delta_{\lambda}=E(\vec \phi_{\omega})-E(\vec u_{\lambda}(t))
\le C_1 |P(\vec u_{\lambda}(t))|$$
for all $t\ge 0$. 
Since $t\mapsto P(\vec u_{\lambda}(t))$ is continuous, 
we see that either 
(i) $P(\vec u_{\lambda}(t))\ge \delta_{\lambda}/C_1$ for all $t\ge 0$, or 
(ii) $P(\vec u_{\lambda}(t))\le -\delta_{\lambda}/C_1$ for all $t\ge 0$. 
Moreover, by Lemma \ref{lem:AP}, we have 
$$\frac{d}{dt} A(\vec u_{\lambda}(t))=P(\vec u_{\lambda}(t))$$
for all $t\ge 0$. 
Therefore, we see that $|A(\vec u_{\lambda}(t))| \to \infty$ as $t\to \infty$. 
This contradicts the fact that 
$|A(\vec u_{\lambda}(t))|\le C_2$ for all $t\ge 0$. 
Hence, $e^{i\omega t}\vec \phi_{\omega}$ is unstable. 
\end{proof}

{\bf Acknowledgment}. 
This work was supported by JSPS KAKENHI Grant Numbers 15K04968, 26247013. 



\begin{thebibliography}{99}

\bibitem{agr}
G. Agrawal, 
Nonlinear fiber optics, 
Optics and Photonics, Academic Press, 2007.

\bibitem{caz}
T. Cazenave,
Semilinear Schr\"odinger equations, 
Courant Lecture Notes in Mathematics 10, 
Amer. Math. Soc., Providende, RI, 2003. 

\bibitem{CCO1}
M. Colin, T. Colin and M. Ohta, 
Stability of solitary waves for a system of  nonlinear Schr\"odinger equations with three wave interaction, 
Ann. Inst. H. Poincar\'{e}, Anal. Non Lin\'{e}aire, {\bf 26} (2009), 2211--2226. 

\bibitem{CCO2}
M. Colin, T. Colin and M. Ohta, 
Instability of standing waves for a system of nonlinear Schr\"odinger equations with three-wave interaction, 
Funkcial. Ekvac. {\bf 52} (2009), 371--380. 

\bibitem{CO3}
M. Colin and M. Ohta, 
Bifurcation from semitrivial standing waves and ground states 
for a system of nonlinear Schr\"odinger equations, 
SIAM J. Math. Anal. {\bf 44} (2012), 206--223. 

\bibitem{GSS1}M. Grillakis, J. Shatah and W. Strauss, 
Stability theory of solitary waves in the presence of symmetry, I, 
J. Funct. Anal., {\bf 74} (1987), 160--197. 

\bibitem{HOT}
N. Hayashi, T. Ozawa and K. Tanaka, 
On a system of nonlinear Schr\"odinger equations with quadratic interaction, 
Ann. Inst. H. Poincar\'e, Anal. Non Lin\'eaire, {\bf 30} (2013), 661--690.

\bibitem{IMW}J. Ieda, T. Miyakawa and M. Wadati, 
Matter-wave solitons in an $F=1$ spinor Bose-Einstein condensate, 
J. Phys. Soc. Jpn., {\bf 73} (2004), 2996--3007. 

\bibitem{KS}
T. Kanna and K. Sakkaravarthi, 
Multicomponent coherently coupled and incoherently coupled solitons and their collisions, 
J. Phys. A: Math. Theor., {\bf 44} (2011), 285211 (30pp). 

\bibitem{maeda}M. Maeda, 
Stability of bound states of Hamiltonian PDEs in the degenerate cases, 
J. Funct. Anal., {\bf 263} (2012), 511--528. 

\bibitem{NW}
N. Nguyen and Z.-Q. Wang, 
Orbital stability of solitary waves for a nonlinear Schr\"odinger system, 
Adv. Differential Equations, {\bf 16} (2011), 977--1000. 

\bibitem{oht96}M. Ohta, 
Stability of solitary waves for coupled nonlinear Schr\"odinger equations, 
Nonlinear Anal., {\bf 26} (1996), 933--939. 

\bibitem{oht11}M. Ohta, 
Instability of bound states for abstract nonlinear Schr\"{o}dinger equations, 
J. Funct. Anal., {\bf 261} (2011), 90--110. 

\bibitem{WC}Z. Wang and S. Cui, 
Multi-speed solitary wave solutions for a coherently coupled nonlinear Schr\"odinger system, 
J. Math. Phys., {\bf 56} (2015), no. 2, 021503, 13 pp. 

\bibitem{wei1}M. I. Weinstein, 
Modulational stability of ground states of nonlinear Schr\"odinger equations, 
SIAM J. Math. Anal., {\bf 16} (1985), 472--491. 

\bibitem{wei2}M. I. Weinstein, 
Lyapunov stability of ground states of nonlinear dispersive evolution equations, 
Comm. Pure Appl. Math., {\bf 39} (1986), 51--68. 

\bibitem{yamazaki}Y. Yamazaki, 
Stability of line standing waves near the bifurcation point for nonlinear Schr\"odinger equations, 
Kodai Math. J., {\bf 38} (2015), 65--96. 

\end{thebibliography}
\end{document}